\documentclass[reqno,dvips]{amsart}

\usepackage{mathrsfs}
\usepackage{amsmath}
\usepackage{amssymb}
\usepackage{cite}
\usepackage{latexsym}
\usepackage{graphicx}

\usepackage{color}

\theoremstyle{plain}
\newtheorem{thm}{Theorem}[section]
\newtheorem{lem}[thm]{Lemma}

\newtheorem{dfn}[thm]{Definition}
\newtheorem{prop}[thm]{Proposition}
\newtheorem{rmk}[thm]{Remark}

\def\C{\mathscr{C}}
\def\D{\mathrm{D}}
\def\G{\mathscr{G}}
\def\T{\mathrm{T}}

\def\d{\mathrm{d}}

\def\Cset{\mathbb{C}}
\def\Eset{\mathbb{E}}
\def\Fset{\mathbb{F}}
\def\Gset{\mathbb{G}}
\def\Kset{\mathbb{K}}
\def\Lset{\mathbb{L}}

\def\Nset{\mathbb{N}}

\def\Qset{\mathbb{Q}}
\def\Rset{\mathbb{R}}
\def\Zset{\mathbb{Z}}

\def\GL{\mathrm{GL}}

\def\epsilon{\varepsilon}

\makeatletter
 \@addtoreset{equation}{section}
\makeatother
\def\theequation{\arabic{section}.\arabic{equation}}

\begin{document}


\title[Nonintegrability of the SEIR epidemic model]%
{Nonintegrability of the SEIR epidemic model}

\author{Kazuyuki Yagasaki}

\address{Department of Applied Mathematics and Physics, Graduate School of Informatics,
Kyoto University, Yoshida-Honmachi, Sakyo-ku, Kyoto 606-8501, JAPAN}
\email{yagasaki@amp.i.kyoto-u.ac.jp}

\date{\today}

\begin{abstract}
We prove the nonintegrability
 of the susceptible-exposed-infected-removed (SEIR) epidemic model
 in the Bogoyavlenskij sense.
This property of the SEIR model is different
 from the more fundamental susceptible-infected-removed (SIR) model,
 which is Bogoyavlenskij-nonintegrable.
Our basic tool for the proof is an extension of the Morales-Ramis theory
 due to Ayoul and Zung.
Moreover, we introduce three new state variables
 and extend the system to a six-dimensional system
 to treat transcendental first integrals and commutative vector fields.
We also use the fact that the incomplete gamma function $\Gamma(\alpha,x)$ is not elementary
 for $\alpha\not\in\Nset$, of which a proof is included.
\end{abstract}
\maketitle


\section{Introduction}

Mathematical epidemic models have attracted a great deal of attention
 in diverse scientific disciplines including epidemiology, medicine, pharmacy and economics
 as well as fundamental ones like mathematics, physics and biology
 \cite{BC19,M15,ST11}.
In particular, the recent intense interest in the models mainly
 results from the pandemic spread of the coronavirus COVID-19
 beginning at the end of 2019 in China.
The \emph{susceptible-infected-removed} (SIR) model 
\begin{align}
\dot{S}=-rSI,\quad
\dot{I}=rSI-a I,\quad
\dot{R}=aI
\label{eqn:sir}
\end{align}
and the \emph{susceptible-exposed-infected-removed} (SEIR) model
\begin{align}
\dot{S}=-rSI,\quad
\dot{E}=rSI-b E,\quad
\dot{I}=b E-a I,\quad
\dot{R}=aI
\label{eqn:seir}
\end{align}
are ones of the most famous mathematical epidemic models \cite{BC19,M15,ST11}.
Here the state variables $S$, $E$, $I$ and $R$ represent
 the numbers of susceptible, exposed, infected and removed individuals, respectively,
 while $r$, $b$ and $a$ represent the infection, latent and removal rates, respectively. 
The former was originally introduced by Kermack and McKendrick \cite{KM27} in 1927
 sometime after the Spanish flu pandemic of 1918-1919.
The latent period $1/b$ is taken into consideration in the latter, unlike the former.

Since the variable $R$ does not appear in the first two equations of \eqref{eqn:sir}
 and in the first three equations of \eqref{eqn:seir},
 the dynamics of \eqref{eqn:sir} and \eqref{eqn:seir}
 are essentially described by the two-dimensional system 
\begin{align}
\dot{S}=-rSI,\quad
\dot{I}=rSI-a I
\label{eqn:si}
\end{align}
and by the three-dimensional system
\begin{align}
\dot{S}=-rSI,\quad
\dot{E}=rSI-b E,\quad
\dot{I}=b E-a I,
\label{eqn:sei}
\end{align}
respectively.
We also refer to \eqref{eqn:si} and \eqref{eqn:sei}, respectively, as the SIR and SEIR models.
We see that
\[
F_1(S,I,R)=S+I+R,\quad
F_2(S,I)=S\exp\left(-\frac{r}{a}(S+I)\right)
\]
are first integrals, i.e., conservative quantities, of \eqref{eqn:sir}.
So solutions to \eqref{eqn:sir} (resp. to \eqref{eqn:si}) continue to remain
 in a one-dimensional level set of $F_1(S,I,R),F_2(S,I)=$\linebreak const.
 (resp. $F_2(S,I)=$const.)
 in the three-dimensional  (resp. two-dimensional) phase space.
Thus, we can obtain all solutions to \eqref{eqn:sir} and to \eqref{eqn:si}.
See, e.g., Section~2.4 of \cite{BC19}, Section~2.1 of \cite{M15} or Appendix~A of \cite{Y22} for more details.
The object of the present paper
 is to prove that this is not the case in \eqref{eqn:seir} and \eqref{eqn:sei}. 
To make our statement precise,
 we use the general concept of integrability due to Bogoyavlenskij \cite{B98}

Consider general $n$-dimensional systems of the form
\begin{equation}
\dot{x}=f(x),\quad\mbox{$x\in\Rset^n$ or $\Cset^n$},
\label{eqn:gsys}
\end{equation}
where $f:\Rset^n\to\Rset^n$ or $\Cset^n\to\Cset^n$ is meromorphic or analytic.

\begin{dfn}[Bogoyavlenskij]
\label{dfn:1a}
The $n$-dimensional system \eqref{eqn:gsys}
 is called \emph{$(\ell,n-\ell)$-integrable} or simply \emph{integrable} 
 if there exist $\ell$ vector fields $f_1(x)(:= f(x)),f_2(x),\linebreak\dots,f_\ell(x)$
 and $n-\ell$ scalar-valued functions $F_1(x),\dots,F_{n-\ell}(x)$ such that
 the following two conditions hold:
\begin{enumerate}
\setlength{\leftskip}{-1.8em}
\item[\rm(i)]
$f_1(x),\dots,f_\ell(x)$ are linearly independent almost everywhere
 and commute with each other,
 i.e., $[f_j,f_k](x):=\D f_k(x)f_j(x)-\D f_j(x)f_k(x)\equiv 0$ for $j,k=1,\ldots,\ell$,
 where $[\cdot,\cdot]$ denotes the Lie bracket$;$
\item[\rm(ii)]
The derivatives $\D F_1(x),\dots, \D F_{n-\ell}(x)$ are linearly independent almost everywhere
 and $F_1(x),\dots,F_{n-\ell}(x)$ are first integrals of $f_1, \dots,f_\ell$,
 i.e., $\D F_k(x)\cdot f_j(x)\equiv 0$ for $j=1,\ldots,\ell$ and $k=1,\ldots,n-\ell$,
 where ``$\cdot$'' represents the inner product.
\end{enumerate}
We say that the system is \emph{rationally} $($resp. \emph{meromorphically}$)$ \emph{integrable}
 if the first integrals and commutative vector fields are rational $($resp. meromorphic$)$. 
\end{dfn}

Definition~\ref{dfn:1a} is considered as a generalization of
 Liouville-integrability for Hamiltonian systems \cite{A89,M99}
 since an $n$-degree-of-freedom Liouville-integrable Hamiltonian system with $n\ge 1$
 has not only $n$ functionally independent first integrals
 but also $n$ linearly independent commutative (Hamiltonian) vector fields
 generated by the first integrals.
The systems \eqref{eqn:sir} and \eqref{eqn:si} are, respectively,
 $(1,2)$- and $(1,1)$-integrable.

We also see that
\begin{equation}
F_1(S,E,I,R)=S+E+I+R,\quad
F_2(S,E,I)=S\exp\left(-\frac{r}{a}(S+E+I)\right)
\label{eqn:F12}
\end{equation}
are first integrals of \eqref{eqn:seir},
 and $F_2(S,E,I)$ is a first integral of \eqref{eqn:sei}.
Obviously, if the system \eqref{eqn:sei} is integrable,
 then so is the system \eqref{eqn:seir}, and vice versa.
Our main result is stated as follows.

\begin{thm}
\label{thm:main}
If $a\neq 0$ and $b/a\not\in\Qset\setminus\{1\}$,
 then the SEIR epidemic model \eqref{eqn:sei}
 is not Bogoyavlenskij-integrable near $S=0$
 such that the first integrals and commutative vector fields
 are rational functions of $S$, $E$, $I$, $e^S$, $e^E$ and $e^I$.
Here we may take $a,b,r\in\Cset$.
\end{thm}

Obviously, the system \eqref{eqn:sei} is Bogoyavlenskij-integrable
 for $a=0$ and/or $b=0$ 
 since $\tilde{F}_1(S,E,I)=S+E+I$ is another first integral for $a=0$
 and the variable $E$ does not appear in its right hand side for $b=0$.
When $a+b=0$, the system \eqref{eqn:sei} is $(1,2)$-integrable,
 since
\[
F_3(S,E,I)=-\frac{a}{r}S+\tfrac{1}{2}aI^2-(a+1)(S+E+I)+\tfrac{1}{2}(S+E+I)^2
\]
is another first integral.
We still conjecture that the SEIR epidemic model \eqref{eqn:sei}
 is Bogoyavlenskij-nonintegrale for $a\neq 0$ and $b/a\in\Qset\setminus\{0,\pm 1\}$,
 but it seems to be beyond the approach used here to prove the conjecture.

Our basic tool to prove Theorem~\ref{thm:main}
 is an extension of the Morales-Ramis theory \cite{M99,MR01}
 due to Ayoul and Zung \cite{AZ10},
 which is based on the differential Galois theory
 for linear differential equations \cite{CH11,PS03}.
The extended theory says that the system ~\eqref{eqn:gsys} is Bogoyavlenskij-nonintegrable
 near a particular nonconstant solution
 if the \emph{identity component} of the \emph{differential Galois group}
 for the \emph{variational equation $(\!$VE$\,)$}, i.e.,
 the linearized equation, of \eqref{eqn:gsys} around the solution is not commutative.
See Sections~2.1 and 2.2 for more details.
So we first obtain nonconstant particular solutions to the SEIR model \eqref{eqn:sei}.
Moreover, we introduce three new state variables (see Eq.~\eqref{eqn:nsv} below)
 and extend the three-dimensional system \eqref{eqn:sei} to a six-dimensional system
 to treat such transcendental first integrals and commutative vector fields as $F_2(S,E,I)$.
We compute the differential Galois group for the VE
 of the extended system around the particular solutions.
For the computation of the differential Galois group,
 we use the fact that the (upper) incomplete gamma function
\[
\Gamma(\alpha,x)=\int_x^\infty y^{\alpha-1} e^{-y}\d y
\]
is not elementary for $\alpha\not\in\Nset\subset\Qset\setminus\{0\}$.
Since it does not seem to be found in references,
 we include a proof of this fact in Section~2.3
 along with the Risch algorithm \cite{B05,GCL92,R69},
 which gives a fundamental tool for determining
 whether integrals of elementary functions are also elementary or not.

The outline of this paper is as follows:
In Section~2 we provide necessary information
 on the differential Galois theory and general Morales-Ramis theory,
 and give a proof for the necessary fact on the incomplete gamma function.
We obtain particular solutions to \eqref{eqn:sei},
 introduce the new state variable and extended system,
 and shortly discuss its VE along the solutions in Section~3.
We give a proof of Theorem~\ref{thm:main} for $a\neq b$ and $a=b$ in Sections~4 and 5, 
 respectively.


\section{Preliminaries}
In this section we provide necessary preliminaries
 for the differential Galois theory, the general Morales theory and incomplete gamma function.

\subsection{Differential Galois theory}

We begin with the differential Galois theory for linear differential equations.
See the textbooks \cite{CH11,PS03} for more details on the theory.

Consider a linear system of differential equations
\begin{equation}
\dot{x}=Ax,\quad A\in\mathrm{gl}(n,\Kset),
\label{eqn:lsys}
\end{equation}
where $\Kset$ is a differential field and
 $\mathrm{gl}(n,\Kset)$ denotes the ring of $n\times n$ matrices
 with entries in $\Kset$.
Here a \emph{differential field} is a field
 endowed with a derivation $\partial$,
 which is an additive endomorphism satisfying the Leibniz rule
 and represented by the overdot in \eqref{eqn:lsys}.
The set $\mathrm{C}_{\Kset}$ of elements of $\Kset$ for which $\partial$ vanishes
 is a subfield of $\Kset$
 and called the \emph{field of constants of $\Kset$}.
In our application of the theory in this paper,
 the field of constants is $\Cset$, which is algebraically closed.

A \emph{differential field extension} $\Lset\supset \Kset$
 is a field extension such that $\Lset$ is also a differential field
 and the derivations on $\Lset$ and $\Kset$ coincide on $\Kset$.
A differential field extension $\Lset\supset \Kset$
 satisfying the following two conditions is called a \emph{Picard-Vessiot extension}
 for \eqref{eqn:lsys}:
\begin{enumerate}
\item[\bf (PV1)]
The field $\Lset$ is generated by $\Kset$
 and elements of a fundamental matrix of \eqref{eqn:lsys};
\item[\bf (PV2)]
The fields of constants for $\Lset$ and $\Kset$ coincide.
\end{enumerate}
The system \eqref{eqn:lsys}
 admits a Picard-Vessiot extension which is unique up to isomorphism.

We now fix a Picard-Vessiot extension $\Lset\supset \Kset$
 and fundamental matrix $\Phi$ with entries in $\Lset$
 for \eqref{eqn:lsys}.
Let $\sigma$ be a \emph{$\Kset$-automorphism} of $\Lset$,
 which is a field automorphism of $\Lset$
 that commutes with the derivation of $\Lset$
 and leaves $\Kset$ pointwise fixed.
Obviously, $\sigma(\Phi)$ is also a fundamental matrix of \eqref{eqn:lsys}
 and consequently there is a matrix $M_\sigma$ with constant entries
 such that $\sigma(\Phi)=\Phi M_\sigma$.
This relation gives a faithful representation
 of the group of $\Kset$-automorphisms of $\Lset$
 on the general linear group as
\[
R\colon \mathrm{Aut}_{\Kset}(\Lset)\to\GL(n,\mathrm{C}_{\Lset}),
\quad \sigma\mapsto M_{\sigma},
\]
where $\GL(n,\mathrm{C}_{\Lset})$
is the group of $n\times n$ invertible matrices with entries in $\mathrm{C}_{\Lset}$.
The image of $R$
 is a linear algebraic subgroup of $\GL(n,\mathrm{C}_{\Lset})$,
 which is called the \emph{differential Galois group} of \eqref{eqn:lsys}
 and often denoted by $\mathrm{Gal}(\Lset/\Kset)$.
This representation is not unique
 and depends on the choice of the fundamental matrix $\Phi$,
 but a different fundamental matrix only gives rise to a conjugated representation.
Thus, the differential Galois group is unique up to conjugation
 as an algebraic subgroup of the general linear group.

Let $\G\subset\GL(n,\mathrm{C}_{\Lset})$ be an algebraic group.
Then it contains a unique maximal connected algebraic subgroup $\G^0$,
 which is called the \emph{connected component of the identity}
 or \emph{identity component}.
The identity component $\G^0\subset\G$ is
 the smallest subgroup of finite index, i.e., the quotient group $\G/\G^0$ is finite.
 
We also use the following properties of the differential Galois group below
 (See, e.g., Propositions 6.1.2(a) and 6.3.4(b) of \cite{CH11} for their proofs).
 
\begin{prop}
\label{prop:G1}
Let $\Lset\supset\hat{\Kset}\supset\Kset$ be fields extensions
 such that $\Lset\supset\Kset$ and $\hat{\Lset}\supset\Kset$
 are Picard-Vessiot extensions for linear systems of the form \eqref{eqn:lsys}
 and the field $\mathrm{C}_{\Kset}$ of constants is algebraically closed.
Then the following hold$\,:$ 
\begin{enumerate}
\setlength{\leftskip}{-1.8em}
\item[\rm(i)]
If $x\in\Lset\setminus\Kset$,
 then there exists $\sigma\in\mathrm{Gal}(\Lset/\Kset)$ such that $\sigma(x)\neq x;$
\item[\rm(ii)]
$\mathrm{Gal}(\Lset/\hat{\Kset})$ is a normal subgroup of $\mathrm{Gal}(\Lset/\Kset)$,
 i.e., $\mathrm{Gal}(\Lset/\hat{\Kset})\subset\mathrm{Gal}(\Lset/\Kset)$
 and $\sigma^{-1}\hat{\sigma}\sigma\in\mathrm{Gal}(\Lset/\hat{\Kset})$
 for any $\hat{\sigma}\in\mathrm{Gal}(\Lset/\hat{\Kset})$ and $\sigma\in\mathrm{Gal}(\Lset/\Kset)$.
\end{enumerate}
\end{prop}

Let the linear system \eqref{eqn:lsys} defined on a Riemann surface $\C$.
A point $\bar{t}\in\C$ is called a \emph{singular point} if $A$ is not bounded at $x$.
A singular point $\bar{t}$ is called \emph{regular}
 if for any sector $\kappa_1<\arg(t-\bar{t})<\kappa_2$ with $\kappa_1<\kappa_2$
 there exists a fundamental matrix $\Phi(t)=(\phi_{ij}(t))$
 such that for some $C>0$ and integer $N$
 $|\phi_{ij}|<C|t-\bar{t}|^N$ as $t\to\bar{t}$ in the sector;
 otherwise it is called \emph{irregular}.

If the Riemann surface $\C$ contains the inifinity $\infty$,
 then we change the independent variable as $\tau=1/t$
 and transform \eqref{eqn:lsys} as
\begin{equation}
\frac{\d}{\d\tau}x=-\frac{1}{\tau^2}Ax.
\label{eqn:lsys1}
\end{equation}
We say that
 the inifinity $\infty$ is a \emph{regular} (resp. \emph{irregular}) \emph{singular point}
 if the origin $\tau=0$ is a regular (resp. irregular) singular point for \eqref{eqn:lsys1}.
 
\subsection{Generalized Morales-Ramis theory}
We next briefly review the extended Morales-Ramis theory for the general system \eqref{eqn:gsys}
 in a necessary setting.
See \cite{AZ10,M99,MR01} for more details on the theory.

Consider the general system \eqref{eqn:gsys}.
Let $x=\phi(t)$ be its nonstationary particular solution.
The VE of \eqref{eqn:gsys} along $x=\phi(t)$ is given by
\begin{equation}
\dot{\xi} = \D f(\phi(t))\xi,\quad \xi \in \Cset^n.
\label{eqn:gve}
\end{equation}
Let $\C$ be a curve given by $x=\phi(t)$
 and let $\overline{\C}$ be its closure containing points at infinity.
Assume that the vector field $f(x)$ can be meromorphically extended
 to a region containing $\overline{\C}$.
We take the meromorphic function field on $\overline{\C}$
 as the coefficient field $\Kset$ of \eqref{eqn:gve}.
Using arguments given by Morales-Ruiz and Ramis \cite{M99,MR01}
 and Ayoul and Zung \cite{AZ10}, we have the following result.

\begin{thm}
\label{thm:MR}
Let $\G$ be the differential Galois group of \eqref{eqn:gve}.
Suppose that the VE \eqref{eqn:gve} has an irregular singularity at infinity. 
If Eq.~\eqref{eqn:gsys} is rationally integrable near $\C$,
 then the identity component $\G^0$ of $\G$ is commutative.
\end{thm}

\begin{rmk}\
\label{rmk:2a}
\begin{enumerate}
\setlength{\leftskip}{-1.8em}
\item[\rm(i)]
In Theorem~$\ref{thm:MR}$,
 if the VE \eqref{eqn:gve} has no irregular singularity at infinity,
 then we can replace the word “rationally” with “meromorphically” in the conclusion.
See Section~$4.2$ of {\rm\cite{M99}} or Section~$5.2$ of {\rm\cite{MR01}}
 for the details.
\item[\rm(ii)]
Ayoul and Zung {\rm\cite{AZ10}} treated the VE \eqref{eqn:gve} only on $\C$,
  but we can easily extend their result to our situation using arguments
  given in Section~$4.2$ of {\rm\cite{M99}} or Section~$5.2$ of {\rm\cite{MR01}},
  although the obtained result is weaker$\,:$ only rational integrability is mentioned
  instead of meromorphic one, as stated above.
See also Section~$2$ of {\rm\cite{ALMP18}}.
\end{enumerate}
\end{rmk}

\subsection{Implicit gamma function}

As our final preliminary we give a proof for the fact that
 the implicit gamma function $\Gamma(\alpha,x)$ is not elementary for $\alpha\not\in\Nset$.
We first explain the Risch algorithm \cite{B05,GCL92,R69}, which is required for the proof.
Our exposition mainly follows Chapter~12 of \cite{GCL92} with some modifications.

Let $\Eset$ be a differential field
 such that its constant field $\mathrm{C}_\Eset$ is algebraically closed.
We say that a differential extension $\Fset$ of $\Eset$ is \emph{elementary}
 if for some $n\in\Nset$ there exists $\theta_1,\ldots,\theta_n\in\Eset$
 such that $\Fset=\Eset(\theta_1,\ldots,\theta_n)$
 and one of the following holds for each $j=1,\ldots,n$:
\begin{enumerate}
\setlength{\leftskip}{-1.6em}
\item[\rm(i)]
$\theta_j$ is algebraic over $\Fset_j$;
\item[\rm(ii)]
$\theta_j'/\theta_j=u'$ for some $u\in\Fset_j$;
\item[\rm(iii)]
$\theta_j'=u'/u$ for some $u\in\Fset_j$,
\end{enumerate}
where $\Fset_j=\Eset(\theta_1,\ldots,\theta_{j-1})$
 and the prime represents  differentiation.
When conditions~(ii) and (iii) hold, respectively,
 we have $\theta_j=e^u$ and $\theta_j=\log u$ informally
 and call $\theta_j$ \emph{exponential} and \emph{logarithmic} over $\Fset_j$. 
If $\Fset$ is an elementary extension of $\Cset(x)$,
 then $\Fset$ is called a \emph{field of elementary functions}.
We have the following classical fundamental result on integration.

\begin{thm}[Liouville]
\label{thm:L}
Let $\Fset$ be a field of elementary functions, and let $\varphi\in\Fset$.
Suppose that $\varphi=\psi'$ for some $\psi\in\Gset$,
 where $\Gset$ is an elementary extension of $\Fset$.
Then there exist $v_0,v_1,\ldots,v_m\in\Fset$ and constants $c_1,\ldots,c_m\in\Cset$
 for some $m\ge 0$ such that
\[
\varphi=v_0'+\sum_{j=1}^mc_j\frac{v_j'}{v_j},
\]
i.e.,
\[
\int \varphi\,\d x=v_0+\sum_{j=1}^mc_j\log(v_j).
\]
\end{thm}

See Section~5.5 of \cite{B05} or Section~12.4 of \cite{GCL92}
 for a proof of Theorem~\ref{thm:L}.

We are now ready to state a necessary case of the Risch algorithm.
Let $\Fset=\Eset(\theta_1,\ldots,\theta_n)$ be an elementary extension of $\Eset=\Cset(x)$
 for some $n\in\Nset$ such that $\theta_n$ is exponential over $\Fset_{n-1}$.
Let $\theta=\theta_n$ and $p(\theta)\in\Fset_{n-1}[\theta]$.
Using Theorem~\ref{thm:L}, we can show that
 if the integral of $\varphi(\theta)$ is an elementary function, then 
\[
\int p(\theta)\,\d x=\sum_{j=0}^kq_j\theta^j+\sum_{j=1}^mc_j\log(v_j),
\]
where $k$ is the order of the polynomial $p$,
 $q_j\in\Fset_{n-1}$, $j=0,\ldots,k$,
 and $v_j\in\Fset_{n-1}$, $j=1,\ldots,m$.
In the above relation, the case of $m=0$,
 which means that the second term does not appear,  is also allowed.
Moreover, we have
\begin{equation}
p_0=q_0'+\sum_{j=1}^mc_j\frac{v_j'}{v_j},\quad
p_j=q_j'+ju'q_j,\quad
j=1,\ldots,k,
\label{eqn:R}
\end{equation}
where $u'=\theta'/\theta$ and $p_0,\ldots,p_k\in\Fset_{n-1}$ such that
\[
p=\sum_{j=0}^kp_j\theta^j.
\]
See Section~12.7 of \cite{GCL92} for the details on derivation of these results.
More general results are found in \cite{B05,GCL92}.
From the first equation of \eqref{eqn:R}
 we see that if $p_0=0$, then $q_0=0$ and $m=0$.
The second equation of \eqref{eqn:R} is also called the \emph{Risch differential equation}
 in this context.

\begin{prop}
\label{prop:R}
When $\alpha\not\in\Nset$, 
 the incomplete gamma function $\Gamma(\alpha,x)$ is not elementary,
 i.e., $\Gamma(\alpha,x)$ does not belong to any field of elementary functions.
\end{prop}

\begin{proof}
We first assume that $-\alpha\in\Zset_{\ge 0}:=\Nset\cup\{0\}$
 and set $\beta=-\alpha+1\in\Nset$.
Let $\theta_1=e^{-x}$.
From the Risch algorithm we see that if $\Gamma(\alpha,x)$ is elementary, then
\begin{equation}
\int x^{-\beta}\theta_1\d x=q\theta_1,
\label{eqn:prop2a0}
\end{equation}
where $q\in\Cset(x)$.

Suppose that there exists such a rational function $q$.
Differentiating both sides of \eqref{eqn:prop2a0}, we obtain
\begin{equation}
q'-q=x^{-\beta}.
\label{eqn:prop2a1}
\end{equation}
Let $q=q_1/q_2$,
 where $q_1,q_2\in\Cset[x]$ have no common factor and $q_2$ is monic.
We rewrite \eqref{eqn:prop2a1} as
\begin{equation}
q_1'q_2-q_1q_2'-q_1q_2=x^{-\beta}q_2^2,
\label{eqn:prop2a2}
\end{equation}
which yields
\[
x^\beta q_1q_2'=q_2(x^\beta q_1'-x^\beta q_1-q_2).
\]
Hence, $x^\beta q_2'$ is divisible by $q_2$,
 so that $q_2$ is expressed as $q_2=x^k$ for some $k\in\Zset_{\ge 0}$.
From \eqref{eqn:prop2a2} we have
\begin{equation}
xq_1'-(x+k)q_1=x^{k-\beta+1},
\label{eqn:prop2a3}
\end{equation}
which means that $k>\beta-1$.
Let
\begin{equation}
q_1=\sum_{j=0}^\ell a_jx^j,\quad
a_\ell\neq 0,
\label{eqn:q1a}
\end{equation}
where $\ell\in\Zset_{\ge 0}$ and $a_j\in\Cset$, $j=0,\ldots,\ell$, with $a_\ell\neq 0$.
Substituting \eqref{eqn:q1a} into \eqref{eqn:prop2a3},
 we see that $\ell+1=k-\beta+1$, i.e., $k=\ell+\beta$,  and
\[
a_\ell=1,\quad
a_0=0,\quad
(j-k)a_j-a_{j-1}=0,\quad
j=1,\ldots,\ell,
\]
which never occurs since $j-k\le-\beta$ for $j\le\ell$.
So there does not exist a rational function $q$ satisfying \eqref{eqn:prop2a0}.

We next assume that $\alpha\in\Cset\setminus\Zset$.
Let $\theta_2=\log x$ and $\theta_3=e^{(\alpha-1)\theta_2}$.
From the Risch algorithm we see that if $\Gamma(\alpha,x)$ is elementary, then
\[
\int\theta_1\theta_3\d x=q\theta_3,
\]
where $q\in\Cset(x,\theta_1,\theta_2)$.

Suppose that there exists such a rational function $q$.
Differentiating both sides of the above equation, we obtain
\begin{equation}
xq'+(\alpha-1) q=x\theta_1.
\label{eqn:prop2a4}
\end{equation}
Let $q=q_1/q_2$,
 where $q_1,q_2\in\Cset(x,\theta_1)[\theta_2]$ have no common factor
 and $q_2$ is monic as in the above case.
We rewrite \eqref{eqn:prop2a4} as
\[
x(q_1'q_2-q_1q_2')+(\alpha-1)q_1q_2=x\theta_1q_2^2,
\]
which yields
\[
xq_1q_2'=q_2(xq_1'+(\alpha-1)q_1-x\theta_1q_2).
\]
Hence, $q_2'$ is divisible by $q_2$, so that $q_2=1$, since $\theta_2'=1/x$.
Let
\begin{equation}
q=q_1=\sum_{j=0}^\ell a_j\theta_2^j,
\label{eqn:q1b}
\end{equation}
where $\ell\in\Zset_{\ge 0}$ and $a_j\in\Cset(x,\theta_1)$, $j=0,\ldots,\ell$,
 with $a_\ell\neq 0$.
Substituting \eqref{eqn:q1b} into \eqref{eqn:prop2a4}, we have
\[
\sum_{j=0}^\ell((xa_j'+(\alpha-1)a_j)\theta_2^j+ja_j\theta_2^{j-1})=x\theta_1.
\]
If $\ell\ge 1$, then
\[
xa_\ell'+(\alpha-1)a_\ell=0,
\]
so that
\[
a_\ell=C x^{1-\alpha},
\]
which means by $\alpha\not\in\Zset$ that $a_\ell\not\in\Cset(x,\theta_1)$,
 where $C\neq 0$ is a constant.
Hence, we have $\ell=0$, so that $q=q_1\in\Cset(x,\theta_1)$.

Let $q=\tilde{q}_1/\tilde{q}_2$,
 where $\tilde{q}_1,\tilde{q}_2\in\Cset(x)[\theta_1]$ have no common factor
  and $\tilde{q}_2$ is monic.
As above, we obtain
\begin{equation}
x\tilde{q}_1\tilde{q}_2'
 =\tilde{q}_2(x\tilde{q}_1'+(\alpha-1)\tilde{q}_1-x\theta_1\tilde{q}_2).
\label{eqn:prop2a5}
\end{equation}
Hence, $\tilde{q}_2'$ is divisible by $\tilde{q}_2$,
 so that $\tilde{q}_2=\theta_1^k$ for some $k\in\Zset_{\ge 0}$,
 since $\theta_1'=-\theta_1$.
From \eqref{eqn:prop2a5} we have
\begin{equation}
x\tilde{q}_1'+(\alpha-1-kx)\tilde{q}_1=x\theta_1^{k+1}.
\label{eqn:prop2a6}
\end{equation}
Let
\begin{equation}
\tilde{q}_1=\sum_{j=0}^\ell a_j\theta_1^j,\quad
a_\ell\neq 0,
\label{eqn:q1c}
\end{equation}
where $\ell\in\Zset_{\ge 0}$ and $a_j\in\Cset(x)$, $j=0,\ldots,\ell$,
 with $a_\ell\neq 0$.
Substituting \eqref{eqn:q1c} into \eqref{eqn:prop2a6}, we have
\[
\sum_{j=0}^\ell(xa_j'+(\alpha-j-1-kx)a_j)\theta_1^j=x\theta_1^k.
\]
Obviously, $\ell\ge k$.
If $\ell>k$, then
\[
xa_\ell'+(\alpha-\ell-1-kx)a_\ell=0,
\]
so that
\[
a_\ell=Cx^{\ell+1-\alpha}e^{kx},
\]
which means that $a_\ell\not\in\Cset(x)$,
 where $C\neq 0$ is a constant.
Hence, we have $\ell=k$ and
\[
xa_j'+(\alpha-j-1-kx)a_j=0,\quad
j=0,\ldots,\ell-1.
\]
By $a_j\in\Cset(x)$,
 we have $a_j=0$, $j=0,\ldots,\ell-1$, and consequently $\tilde{q}_1=a_\ell\theta_1^\ell$.
Since $\tilde{q}_1,\tilde{q}_2$ have no common factor,
 we obtain $\ell,k=0$,
 so that $q$ is a constant and Eq.~\eqref{eqn:prop2a4} never holds.
We complete the proof.
\end{proof}

\begin{rmk}
In references $($e.g., {\rm\cite{R72})}, a proof for 
 a special case of the gamma function,
\[
\Gamma(0,x)=\int_x^\infty \frac{e^{-x}}{x}\,\d x,
\]
which is also called the \emph{exponential function}, is found.
From the fact, it follows by the recurrence formula
\[
\Gamma(\alpha+1,x)=-x^\alpha e^{-x}+(\alpha+1)\Gamma(\alpha,x)
\]
that $\Gamma(\alpha,x)$ is not elementary for $-\alpha\in\Nset$.
\end{rmk}


\section{Particular Solutions and the Variational Equations}

We return to the SEIR system \eqref{eqn:sei}.
We easily see that the $(E,I)$-plane $\{(S,E,I)\mid S=0\}$ is invariant in \eqref{eqn:sei}.
Letting $S=0$ in \eqref{eqn:sei}, we have
\begin{align}
\dot{E}=-b E,\quad
\dot{I}=b E-a I.
\label{eqn:ei}
\end{align}
The general solution to \eqref{eqn:ei} is given by
\[
E=C_1e^{-bt},\quad
I=\frac{bC_1}{a-b}e^{-bt}+\left(C_2-\frac{bC_1}{a-b}\right)e^{-a t},
\]
when $a\neq b$, and
\[
E=C_1e^{-at},\quad
I=aC_1te^{-at}+C_2e^{-at},
\]
when $a=b$, where $C_1,C_2\in\Cset$ are any constants.
Hence, Eq.~\eqref{eqn:sei} has particular solutions
\begin{equation}
(S,E,I)=\left(0,C_1e^{-bt},\bar{I}(t)\right),
\label{eqn:sol1}
\end{equation}
where 
\begin{equation}
\bar{I}(t)=\frac{bC_1}{a-b}e^{-bt}+\left(C_2-\frac{bC_1}{a-b}\right)e^{-a t}
\label{eqn:solI1}
\end{equation}
for $a\neq b$ and
\begin{equation}
\bar{I}(t)=aC_1te^{-at}+C_2e^{-at}
\label{eqn:solI2}
\end{equation}
for $a=b$.
Henceforth we take $C_1=0$,
 which we will also assume in Sections~4 and 5.
In particular, Eqs.~\eqref{eqn:solI1} and \eqref{eqn:solI2} reduce to
\begin{equation}
\bar{I}(t)=C_2e^{-a t}.
\label{eqn:solI}
\end{equation}

Now we introduce the three new state variables
\begin{equation}
X=\exp\left(\frac{r}{a}(S+E+I)\right),\quad
Y=\exp\left(\frac{r}{a}(S+E)\right),\quad
Z=\exp\left(\frac{r}{a}S\right)
\label{eqn:nsv}
\end{equation}
and extend the SEIR system \eqref{eqn:sei} to
\begin{equation}
\begin{split}
&
\dot{S}=-rSI,\quad
\dot{E}=rSI-b E,\quad
\dot{I}=b E-a I,\\
&
\dot{X}=-rIX,\quad
\dot{Y}=-\frac{rb}{a}EY,\quad
\dot{Z}=-\frac{r^2}{a}SIZ.
\end{split}
\label{eqn:seixyz}
\end{equation}
We easily see that the three vector fields
\begin{align*}
(0,0,0,X,0,0),\quad
(0,0,0,0,Y,0),\quad
(0,0,0,0,0,Z)
\end{align*}
are commutative with each other along with \eqref{eqn:seixyz}.
Thus, we have the following.

\begin{prop}
\label{prop:3a}
If the system \eqref{eqn:sei} is Bogoyavlenskij-integrable
 such that the first integrals and commutative vector fields are rational functions
 of $S$, $E$, $I$, $e^S$, $e^E$ and $e^I$,
 then the extended system \eqref{eqn:seixyz} is rationally Bogoyavlenskij-integrable.
\end{prop}

\begin{rmk}
We can apply arguments similar to ones given below for the SEIR system \eqref{eqn:sei} directly 
 but only obtain a very weak result$\,:$
 it is not Bogoyavlenskij-integrable
 such that the first integrals and commutative vector fields are rational functions
 of $S$, $E$ and $I$ unlike $F_2(S,E,I)$.
So we need the above extension
 to treat such transcendental first integrals and commutative vector fields.
\end{rmk}
 
Hereafter, based on Proposition~\ref{prop:3a},
 we prove that the extended system \eqref{eqn:seixyz}
 is not rationally Bogoyavlenskij-integrable
 for the proof of Theorem~\ref{thm:main}.

The system \eqref{eqn:seixyz} has particular solutions
\begin{equation}
(S,E,I,X,Y,Z)=\left(0,0,\bar{I}(t),\bar{X}(t),1,1\right),
\label{eqn:solxyz}
\end{equation}
corresponding to \eqref{eqn:solI1} or \eqref{eqn:solI2}, where
\begin{equation}
\bar{X}(t)
 =\exp\left(\frac{rC_2}{a}e^{-a t}\right).
\label{eqn:solX1}
\end{equation}
The VE of \eqref{eqn:seixyz} around the solution \eqref{eqn:solxyz} is given by
\begin{equation}
\begin{split}
&
\delta\dot{S}=-r\bar{I}(t)\delta S,\quad
\delta\dot{E}=r\bar{I}(t)\delta S-b \delta E,\quad
\delta\dot{I}=b\delta E-a\delta I,\\
&
\delta\dot{X}=-r\bar{X}(t)\delta I-r\bar{I}(t)\delta X,\quad
\delta\dot{Y}=-\frac{rb}{a}\delta E,\quad
\delta\dot{Z}=-\frac{r^2}{a}\bar{I}(t)\delta S.
\end{split}
\label{eqn:ve}
\end{equation}
Let
\[
C_3=\exp\left(-\frac{rC_2}{a}\right).
\]
We solve the initial value problem of \eqref{eqn:ve} and obtain the following:
\begin{equation}
\begin{split}
&
\delta S=C_3\bar{X}(t),\quad
\delta E=rC_2C_3e^{-bt}\int_0^t e^{-(a-b)\tau}\bar{X}(\tau)\d\tau,\\
&
\delta I=be^{-at}\int_0^t\delta E(\tau)e^{a\tau}\d\tau,\quad
\delta X=-r\bar{X}(t)\int_0^t\delta I(\tau)\d\tau,\\
&
\delta Y=-\frac{rb}{a}\int_0^t\delta E(\tau)\d\tau,\quad
\delta Z=-\frac{r^2C_2C_3}{a}\int_0^t e^{-a\tau}\bar{X}(\tau)\d\tau
\end{split}
\label{eqn:solve1}
\end{equation}
when $(\delta S(0),\delta E(0),\delta I(0),\delta X(0),\delta Y(0),\delta Z(0))=(1,0,0,0,0,0)$;
\begin{equation}
\begin{split}
&
\delta S=0,\quad
\delta E=e^{-bt},\quad
\delta I=\frac{b}{a-b}(e^{-bt}-e^{-at}),\\
&
\delta X=\frac{r}{a}\left(\frac{a}{a-b}e^{-bt}-\frac{b}{a-b}e^{-at}-1\right)\bar{X}(t),\\
&
\delta Y=\frac{r}{a}(e^{-bt}-1),\quad
\delta Z=0
\end{split}
\label{eqn:solve2a}
\end{equation}
for $a\neq b$ and
\begin{equation}
\begin{split}
&
\delta S=0,\quad
\delta E=e^{-at},\quad
\delta I=ate^{-at},\\
&
\delta X=\frac{r}{a}((at+1)e^{-at}-1)\bar{X}(t),\quad
\delta Y=\frac{r}{a}(e^{-at}-1),\quad
\delta Z=0
\end{split}
\label{eqn:solve2b}
\end{equation}
for $a=b$ when $(\delta S(0),\delta E(0),\delta I(0),\delta X(0),\delta Y(0),\delta Z(0))=(0,1,0,0,0,0)$;
\begin{equation}
\begin{split}
&
\delta S=0,\quad
\delta E=0,\quad
\delta I=e^{-at},\\
&
\delta X=\frac{r}{a}(e^{-at}-1)\bar{X}(t),\quad
\delta Y=0,\quad
\delta Z=0
\end{split}
\label{eqn:solve3}
\end{equation}
when $(\delta S(0),\delta E(0),\delta I(0),\delta X(0),\delta Y(0),\delta Z(0))=(0,0,1,0,0,0)$;
\begin{equation}
\begin{split}
&
\delta S=0,\quad
\delta E=0,\quad
\delta I=0,\\
&
\delta X=C_3\bar{X}(t),\quad
\delta Y=0,\quad
\delta Z=0
\end{split}
\label{eqn:solve4}
\end{equation}
when $(\delta S(0),\delta E(0),\delta I(0),\delta X(0),\delta Y(0),\delta Z(0))=(0,0,0,1,0,0)$;
\begin{equation}
\delta S=0,\quad
\delta E=0,\quad
\delta I=0,\quad
\delta X=0,\quad
\delta Y=1,\quad
\delta Z=0
\label{eqn:solve5}
\end{equation}
when $(\delta S(0),\delta E(0),\delta I(0),\delta X(0),\delta Y(0),\delta Z(0))=(0,0,0,0,1,0)$; and
\begin{equation}
\delta S=0,\quad
\delta E=0,\quad
\delta I=0,\quad
\delta X=0,\quad
\delta Y=0,\quad
\delta Z=1
\label{eqn:solve6}
\end{equation}
when $(\delta S(0),\delta E(0),\delta I(0),\delta X(0),\delta Y(0),\delta Z(0))=(0,0,0,0,0,1)$.
We easily see that the VE \eqref{eqn:ve} has an irregular singularity at infinity
 since as $\tau\to -0$
\[
\frac{1}{\tau^2}\bar{I}(1/\tau),\frac{1}{\tau^2}\bar{X}(1/\tau)\to\infty
\]
and e.g., $\delta S(1/\tau)\to\infty$ hyper-exponentially in \eqref{eqn:solve1}.

In the next two sections, we take an appropriate differential field $\Kset$ for the VE \eqref{eqn:ve}
 and write its Picard-Vessiot extension and the differential Galois group
 as $\Lset$ and $\G:=\mathrm{Gal}(\Lset/\Kset)$, respectively.


\section{Proof of Theorem~\ref{thm:main} for $a\neq b$}

We prove Theorem~\ref{thm:main} for $a\neq b$.
Let $\Phi(t)$ denote a fundamental matrix of \eqref{eqn:ve}
 of which the $j$th column vector $\Phi_j(t)$
 is given by one of \eqref{eqn:solve1}, \eqref{eqn:solve2a} and \eqref{eqn:solve3}-\eqref{eqn:solve6}
 for $j=1,\ldots,6$.
We take $C_1=0$ in \eqref{eqn:solI1}
 and assume that $b/a\not\in\Qset$.

Let $\Kset=\Cset(e^{-at},\bar{X}(t))$ and let $\sigma\in\G$.
Obviously, by \eqref{eqn:solve3}-\eqref{eqn:solve6}, $\sigma(\Phi_j(t))=\Phi_j(t)$, $j\ge 3$.
We next consider $\Phi_2(t)$.
We first show the following lemma.

\begin{lem}
\label{lem:4a}
We have $e^{-kbt}\not\in\Kset$ for any $k\in\Nset$.
\end{lem}

\begin{proof}
We assume that $e^{-kbt}\in\Kset$ for $k\in\Nset$, and write
\begin{equation}
e^{-kbt}=\frac{g_1}{g_2},
\label{eqn:lem4a1}
\end{equation}
where $g_1,g_2\in\Cset(e^{-at})[\bar{X}(t)]$, $g_1,g_2$ have no common factor and $g_2$ is monic. 
Differentiating both sides of the above equation yields
\[
-kbe^{-kbt}=\frac{\dot{g}_1g_2-g_1\dot{g}_2}{g_2^2}=-kb\frac{g_1}{g_2},
\]
so that
\[
g_1\dot{g}_2=g_2(\dot{g}_1+kbg_1).
\]
Hence, $\dot{g}_2$ is divisible by $g_2$ since by \eqref{eqn:solI} and \eqref{eqn:ve},
\[
\dot{\bar{X}}(t)=-rC_2e^{-at}\bar{X}(t),
\]
which means that $\dot{g}_1,\dot{g}_2\in\Cset(e^{-at})[\bar{X}(t)]$.

Let
\begin{equation}
g_2=\sum_{j=0}^\ell h_j\bar{X}(t)^j
\label{eqn:lem4a2}
\end{equation}
where $\ell\in\Zset_{\ge 0}$ and $h_j\in\Cset(e^{-at})s$, $j=0,\ldots,\ell$, with $h_\ell=1$.
Then we compute
\[
\dot{g}_2
=\sum_{j=0}^\ell(\dot{h}_j-jrC_2h_je^{-at})\bar{X}(t)^j
\]
Since $\dot{g}_2$ is divisible by $g_2$, we have
\begin{equation}
\dot{g}_2=-\ell rC_2g_2.
\label{eqn:lem4a3}
\end{equation}
Substituting \eqref{eqn:lem4a2} into the above equation yields
\[
\sum_{j=0}^\ell(\dot{h}_j-jrC_2h_je^{-at})\bar{X}(t)^j=-\ell rC_2\sum_{j=0}^\ell h_j\bar{X}(t)^j.
\]
Assume that $\ell\ge 1$.
Then for $j=0,\ldots,\ell-1$
\begin{align*}
\dot{h}_j=(je^{-at}-\ell) rC_2 h_j,
\end{align*}
so that
\[
h_j=C_{4j}e^{-\ell rC_2t}\exp\left(-\frac{jrC_2}{a}e^{-at}\right),
\]
where $C_{4j}\in\Cset$ is a constant.
So $C_{4j}=0$ for $j=1,\ldots,\ell-1$.
If $\ell rC_2/a\not\in\Zset$,
 then $C_{40}=0$ and $g_2=\bar{X}(t)^\ell$, since $e^{-\ell rC_2t}\not\in\Cset(e^{-at})$.
If $m=\ell rC_2/a\in\Zset$,
 then $g_2=\bar{X}(t)^\ell+C_{40}e^{-mat}$.
Hence, Eq.~\eqref{eqn:lem4a3} never holds and a contradiction occurs.
Thus, $\ell=0$ and $g_2=1$.

Let
\begin{equation}
g_1=\sum_{j=0}^\ell h_j\bar{X}(t)^j
\label{eqn:lem4a4}
\end{equation}
where $\ell\in\Zset_{\ge 0}$ and $h_j\in\Cset(e^{-at})$, $j=0,\ldots,\ell$, with $h_\ell\neq 0$.
By \eqref{eqn:lem4a1} we have
\[
\dot{g}_1=-kbe^{-kbt}=-kbg_1.
\]
Substituting \eqref{eqn:lem4a4} into the above equation yields
\[
\sum_{j=0}^\ell(\dot{h}_j-jrC_2h_je^{-at})\bar{X}(t)^j=-kb\sum_{j=0}^\ell h_j\bar{X}(t)^j.
\]
In particular,
\[
\dot{h}_\ell=(\ell rC_2e^{-at}-kb)h_\ell,
\]
so that
\[
h_\ell=C_6e^{-kbt}\exp\left(-\frac{\ell rC_2}{a}e^{-at}\right)\not\in\Cset(e^{-at}),
\]
where $C_6\neq 0$ is a constant.
Thus, we have  a contradiction and obtain the desired result.
\end{proof}

Let $\Phi_2(t)=(\delta S(t),\delta E(t),\delta I(t),\delta X(t),\delta Y(t),\delta Z(t))^\T$,
 where the superscript {\scriptsize T} represents transposition.
Since
\[
\frac{\d}{\d t}\sigma(e^{-bt})=-b\sigma(e^{-bt}),
\]
we have
\begin{equation}
\sigma(e^{-bt})=ce^{-bt},
\label{eqn:4a}
\end{equation}
where $c\in\Cset^\ast$.
If $c$ is a $k$-th root of unity, then
\[
\sigma(e^{-kbt})=\sigma(e^{-bt})^k=(c(e^{-bt}))^k=e^{-kbt}.
\]
Hence, by Proposition~\ref{prop:G1}(i) and  Lemma~\ref{lem:4a},
 $c$ is not a root of unity for some $\sigma\in\G$.
From \eqref{eqn:solve2a} and \eqref{eqn:4a} we have
\[
\sigma(\delta E(t))=c\delta E(t),\quad
\sigma(\delta I(t))=\gamma(ce^{-bt}-e^{-at})=c\delta I(t)+\gamma(c-1)e^{-at},
\]
where $\gamma=b/(a-b)$.
Moreover,
\begin{align*}
\sigma(\delta X(t))
=&\frac{r}{a}\left(\frac{ac}{a-b}e^{-bt}-\frac{b}{a-b}e^{-at}-1\right)\bar{X}(t)\\
=&c\delta\bar{X}(t)+\frac{\gamma(c-1)r}{a}(e^{-at}-1)\bar{X}(t)
+\frac{(\gamma+1)(c-1)r}{a}\bar{X}(t)
\end{align*}
and
\[
\sigma(\delta Y(t))=\frac{r}{a}(ce^{-bt}-1)
 =c\delta Y(t)+\frac{(c-1)r}{a}.
\]
Thus, we see that
\[
\sigma(\Phi_2(t))=c\Phi_2(t)+\gamma(c-1)\Phi_3(t)
 +\frac{(\gamma+1)(c-1)r}{aC_3}\Phi_4(t)+\frac{(c-1)r}{a}\Phi_5(t),
\]
where $c$ is not a root of unity for some $\sigma\in\G$.
We have the following for $\Phi_1(t)$.

\begin{lem}
\label{lem:4b}
Let $\sigma\in\G$.
We have
\[
\sigma(\Phi_1(t))=\Phi_1(t)+\alpha_1\Phi_2(t)+\alpha_2\Phi_3(t)+\alpha_3\Phi_4(t)
 +\alpha_4\Phi_5(t)+\alpha_5\Phi_6(t)
\]
for some $\alpha_j\in\Cset$, $j=1,\ldots,5$.
\end{lem}

\begin{proof}
Let $\Phi_1(t)=(\delta S(t),\delta E(t),\delta I(t),\delta X(t),\delta Y(t),\delta Z(t))^\T$.
Obviously, $\sigma(\delta S(t))=\delta S(t)$ for any $\sigma\in\G$.
From \eqref{eqn:ve} we have
\[
\frac{\d}{\d t}\sigma(\delta(E(t))=r\bar{I}(t)\sigma(\delta S(t))-b\sigma(\delta E(t))
=r\bar{I}(t)\delta S(t)-b\sigma(\delta E(t)),
\]
so that by \eqref{eqn:solve1}
\begin{equation}
\sigma(\delta E(t))=re^{-bt}\int_0^t\bar{I}(\tau)\delta S(\tau)e^{b\tau}\d\tau+\alpha_1e^{-bt}
 =\delta E(t)+\alpha_1e^{-bt},
\label{eqn:dE}
\end{equation}
where $\alpha_1\in\Cset$ is a constant.
Similarly, by \eqref{eqn:ve} and \eqref{eqn:dE}
\[
\frac{\d}{\d t}\sigma(\delta I(t))
 =b\sigma(\delta E(t))-a\sigma(\delta I(t))
 =b(\delta E(t)+\alpha_1e^{-bt})-a\sigma(\delta I(t))
\]
so that by \eqref{eqn:solve1}
\begin{align}
\sigma(\delta I(t))
=&be^{-at}\int_0^t\delta E(\tau)e^{a\tau}\d\tau
 +\alpha_1\gamma e^{-bt}+\tilde{\alpha}_2e^{-at}\notag\\
=&\delta I(t)+\alpha_1\gamma(e^{-bt}-1)+\tilde{\alpha}_2e^{-at},
\label{eqn:dI}
\end{align}
and by \eqref{eqn:ve} and \eqref{eqn:dI}
\begin{align*}
\frac{\d}{\d t}\sigma(\delta X(t))
=&-r\bar{X}(t)\sigma(\delta I(t))-r\bar{I}(t)\sigma(\delta X(t))\\
=&-r\bar{X}(t)(\delta I(t)+\alpha_1\gamma e^{-bt}+\tilde{\alpha}_2e^{-at})
 -\bar{I}(t)\sigma(\delta X(t)), 
\end{align*}
so that by \eqref{eqn:solve1}
\begin{align*}
\sigma(\delta X(t))
=& \bar{X}(t)\left(
 -r\int_0^t(\delta I(\tau)+\alpha_1\gamma e^{-bt}+\tilde{\alpha}_2e^{-at})\d\tau
 +\tilde{\alpha}_3\right)\\
=&\delta X(t)+\left(\frac{r\alpha_1\gamma}{b}(e^{-bt}-1)
 +\frac{r\tilde{\alpha}_2}{a}(e^{-at}-1)+\tilde{\alpha}_3\right)\bar{X}(t),
\end{align*}
where $\tilde{\alpha}_2,\tilde{\alpha}_3\in\Cset$ are constants..
Moreover, by \eqref{eqn:ve}, \eqref{eqn:solve1} and \eqref{eqn:dE} 
\begin{align*}
\sigma(\delta Y(t))=\delta Y(t)+\frac{r\alpha_1}{a}(e^{-bt}-1)+\alpha_4,\quad
\sigma(\delta Z(t))=\delta Z(t)+\alpha_5,
\end{align*}
where $\alpha_4,\alpha_5\in\Cset$ are constants.
Letting $\alpha_2=\tilde{\alpha}_2+\alpha_1\gamma$ and $\alpha_3=\tilde{\alpha}_3/C_3$,
 we obtain the desired result.
\end{proof}

Thus, we have
\[
\sigma(\Phi(t))=\Phi(t)
\begin{pmatrix}
1 &  0 & 0 & 0\\
\alpha_1 & c & 0 & 0 & 0 & 0\\
\alpha_2 & \gamma(c-1) & 1 & 0 & 0 & 0\\
\alpha_3 & \gamma_1(c-1)  & 0 & 1 & 0 & 0\\
\alpha_4 & \gamma_2(c-1)  & 0 & 0 & 1 & 0\\
\alpha_5 & 0  & 0 & 0 & 0 & 1\\
\end{pmatrix},
\]
where $\gamma_1=(\gamma+1)r/aC_3$, $\gamma_2=r/a$
 and $c$ is not a root of unity for some $\sigma\in\G$.

\begin{lem}
\label{lem:4c}
There exists $\sigma\in\G$ such that each of the following conditions holds$\,:$
\begin{enumerate}
\setlength{\leftskip}{-1.8em}
\item[\rm(i)]
$\alpha_1\neq 0;$
\item[\rm(ii)]
$\alpha_1=0$ and $c$ is not a root of unity.
\end{enumerate}
\end{lem}

\begin{proof}
Let $\Phi_1(t)=(\delta S(t),\delta E(t),\delta I(t),\delta X(t),\delta Y(t),\delta Z(t))^\T$,
 as in the proof of Lemma~\ref{lem:4b}.
We begin with part~(i).
Substituting \eqref{eqn:solX1} into the second equation of \eqref{eqn:solve1}, we have
\begin{align}
\delta E(t)=&rC_2C_3e^{-bt}
 \int_1^x e^{-(a-b)\tau}\exp\left(\frac{rC_2}{a}e^{-at}\right)\d\tau\notag\\
 =& -\frac{rC_2C_3}{a}x^{b/a}\int_1^x y^{-b/a}\exp\left(\frac{rC_2}{a}y\right)\d y
\label{eqn:4b}
\end{align}
where $x=e^{-at}$.
Using Proposition~\ref{prop:R},
 we see that $\delta E(t)$ is not an elementary function of  $e^{-at}$,
 so that $\delta E(t)\not\in\Kset$, since $1-b/a\not\in\Nset$.
Hence,  for some $\sigma\in\G$
 we have $\sigma(\delta E(t))\neq\delta E(t)$ by Proposition~\ref{prop:G1}(i)
 and $\alpha_1\neq 0$ by \eqref{eqn:dE}.
Thus we obtain part~(i).

We turn to part~(ii).
Let $\hat{\Kset}=\Kset(\delta E(t))$.
Since $\hat{\Kset}\supset\Kset$ is a Picard-Vessiot extension
 for the linear system consisting of the first and second equations of \eqref{eqn:ve},
 we have $\mathrm{Gal}(\Lset/\hat{\Kset})\subset\G$ by Proposition~\ref{prop:G1}(ii).
By \eqref{eqn:dE} $\alpha_1=0$ for any $\sigma\in\mathrm{Gal}(\Lset/\hat{\Kset})$.
We claim that $e^{-bt}\not\in\hat{\Kset}$ to complete the proof via Proposition~\ref{prop:G1}(i).

We proceed as in the proof of Lemma~\ref{lem:4a} to show the claim.
Suppose that
\begin{equation}
e^{-bt}=\frac{g_1}{g_2}
\label{eqn:lem4c1}
\end{equation}
for some $g_1,g_2\in\Kset[\delta E(t)]$
 such that $g_1,g_2$ have no common factor and $g_2$ is monic.
Differentiating both sides of the above equation yields
\[
-be^{-bt}=\frac{\dot{g}_1g_2-g_1\dot{g}_2}{g_2^2}=-b\frac{g_1}{g_2},
\]
so that
\[
g_1\dot{g}_2=g_2(\dot{g}_1+bg_1).
\]
Hence, $\dot{g}_2$ is divisible by $g_2$
 since by \eqref{eqn:solI}, \eqref{eqn:ve} and \eqref{eqn:solve1}
\[
\delta\dot{E}(t)=rC_2C_3e^{-at}\bar{X}(t)-b\delta E(t),
\]
which means that $\dot{g}_1,\dot{g}_2\in\Kset[\delta E(t)]$.

Let
\[
g_2=\sum_{j=0}^\ell g_j\delta E(t)^j,
\]
where $\ell\in\Zset_{\ge 0}$ and $h_j\in\Kset$, $j=0,\ldots,\ell$, with $h_\ell=1$.
Then we compute
\begin{align*}
\dot{g}_2
=&\sum_{j=0}^\ell \dot{h}_j(\delta E(t)^j
 +jh_j\delta E(t)^{j-1}(rC_2C_3e^{-at}\bar{X}(t)-b\delta E(t)))\\
=&\sum_{j=0}^\ell(\dot{h}_j-jbh_j)\delta E(t)^j
 +rC_2C_3e^{-at}\bar{X}(t)\sum_{j=0}^{\ell-1}(j+1)h_{j+1}\delta E(t)^j.
\end{align*}
Since $\dot{g}_2$ is divisible by $g_2$, we have
\[
\dot{g}_2=-\ell bg_2,
\]
so that
\[
g_2=C_4 e^{-\ell bt},
\]
where $C_4\neq 0$ is a constant.
Hence, by \eqref{eqn:lem4c1}
\[
g_2=C_4\left(\frac{g_1}{g_2}\right)^\ell,
\] 
which means that $g_2=1$ since $g_1,g_2$ have no common factor.

Let
\begin{equation}
g_1=\sum_{j=0}^\ell h_j\delta E(t)^j,
\label{eqn:lem4c2}
\end{equation}
where $\ell\in\Zset_{\ge 0}$ and $h_j\in\Kset$, $j=0,\ldots,\ell$, with $h_\ell\neq 0$.
By \eqref{eqn:lem4c1} we have
\[
\dot{g}_1=-be^{-bt}=-bg_1.
\]
Substituting \eqref{eqn:lem4c2} into the above equation yields
\begin{align}
&\sum_{j=0}^\ell(\dot{h}_j-jbh_j)\delta E(t)^j
 +rC_2C_3e^{-at}\bar{X}(t)\sum_{j=0}^{\ell-1}(j+1)h_{j+1}\delta E(t)^j\notag\\
&=-b\sum_{j=0}^\ell h_j\delta E(t)^j.&
\label{eqn:lem4c3}
\end{align}
In particular,
\[
\dot{h}_\ell=(\ell-1)bh_\ell,
\]
so that
\[
h_\ell=C_5e^{(\ell-1)bt},
\]
where $C_5\neq 0$ is a constant.
If $\ell\neq 1$, then $h_\ell\not\in\Kset$ by Lemma~\ref{lem:4a}.
Hence, we have $\ell=1$ and by \eqref{eqn:lem4c3}
\[
\dot{h}_1=0,\quad
\dot{h}_0+bh_0=-rC_2C_3C_5e^{-at}\bar{X}(t)h_1,
\]
so that $h_1\neq 0$ is a constant and by \eqref{eqn:4b}
\[
h_0=-rC_2C_3C_5h_1e^{-bt}\left(\int_0^te^{-(a-b)\tau}\bar{X}(\tau)\d\tau+C_6\right)
=-C_5h_1\delta E(t)\not\in\Kset.
\]
So we have a contradiction and complete the proof.
\end{proof}

Let
\[
A=
\begin{pmatrix}
1 &  0 & 0 & 0 & 0 & 0\\
\alpha_1 & c & 0 & 0 & 0 & 0\\
\alpha_2 & \gamma(c-1) & 1 & 0 & 0 & 0\\
\alpha_3 & \gamma_1(c-1) & 0 & 1 & 0 & 0\\
\alpha_4 & \gamma_2(c-1) & 0 & 0 & 1 & 0\\
\alpha_5 & 0 & 0 & 0 & 0 & 1
\end{pmatrix}
\]
with $c\neq 1$.
We compute
\[
A^2=
\begin{pmatrix}
1 &0 & 0 & 0 & 0 & 0\\
\alpha_1(c+1) & c^2 & 0 & 0 & 0 & 0\\
\gamma\alpha_1(c-1)+2\alpha_2 & \gamma(c^2-1) & 1 & 0 & 0 & 0\\
\gamma_1\alpha_1(c-1)+2\alpha_3 & \gamma_1(c^2-1) & 0 &1 & 0 & 0\\
\gamma_2\alpha_1(c-1)+2\alpha_3 & \gamma_2(c^2-1) &  0 & 0 & 1 & 0\\
2\alpha_5 & 0 & 0 & 0 & 0 & 1
\end{pmatrix}
\]
and show by induction that for any $k\in\Nset$
\begin{align*}
A^k=&
\begin{pmatrix}
1 & 0 & 0 & 0 & 0 & 0\\[1ex]
\displaystyle
\frac{\alpha_1}{c-1}(c^k-1) &c^k & 0 & 0 & 0 & 0\\[2ex]
\displaystyle
\frac{\gamma\alpha_1}{c-1}(c^k-1)+k(\alpha_2-\gamma\alpha_1)
 & \gamma(c^k-1) & 1 & 0 & 0 & 0\\[2ex]
\displaystyle
\frac{\gamma_1\alpha_1}{c-1}(c^k-1)+k(\alpha_3-\gamma_1\alpha_1)
 & \gamma_1(c^k-1) & 0 & 1 & 0 & 0\\[2ex]
\displaystyle
\frac{\gamma_2\alpha_1}{c-1}(c^k-1)+k(\alpha_3-\gamma_2\alpha_1)
 & \gamma_2(c^k-1) & 0 & 0 & 1 & 0\\[2ex]
k\alpha_5 & 0 & 0 & 0 & 0 & 1
\end{pmatrix}.
\end{align*}
Hence, by Lemmas~\ref{lem:4b} and \ref{lem:4c}, we see that
\begin{align*}
\G^0\supset
&\left\{\left.
\begin{pmatrix}
1 & 0 & 0 & 0 & 0 & 0\\
0 & c & 0 & 0 & 0 & 0\\
\ast & \gamma(c-1) & 1 & 0 & 0 & 0\\
\ast & \gamma_1(c-1) & 0 & 1 & 0 & 0\\
\ast & \gamma_2(c-1)& 0 & 0 & 1 & 0\\
\ast & 0 & 0 & 0 & 0 & 1\end{pmatrix}\,\right|\,c\in\Cset
\right\}\\
&\cup\left\{\left.
\begin{pmatrix}
1 & 0 & 0 & 0 & 0 & 0\\
\alpha_1(c-1) & c & 0 & 0 & 0 & 0\\
\ast & \gamma(c-1) & 1 & 0 & 0 & 0\\
\ast & \gamma_1(c-1) & 0 & 1 & 0 & 0\\
\ast & \gamma_2(c-1) & 0 & 0 & 1 & 0\\
\ast & 0 & 0 & 0 & 0 & 1
\end{pmatrix}\,\right|\,c\in\Cset
\right\},
\end{align*}
where $\alpha_1\neq 0$ is a constant
 and the asterisk represents certain constants which may be zero.
Thus, $\G^0$ is not commutative.
Applying Theorem~\ref{thm:MR} to the extended system \eqref{eqn:seixyz}
 and using Proposition~\ref{prop:3a},
 we obtain the desired result for $b/a\not\in\Qset$.
\qed


\section{Proof of Theorem~\ref{thm:main} for $a=b$}
Finally, we prove Theorem~\ref{thm:main} for $a=b$.
Let $\Phi(t)$ denote a fundamental matrix of \eqref{eqn:ve}
 of which the $j$th column vector $\Phi_j(t)$ is given
 by one of \eqref{eqn:solve1} and \eqref{eqn:solve2b}-\eqref{eqn:solve6}.
 as in Section~4.
We take $C_1=0$ in \eqref{eqn:solI2}.

Let $\Kset=\Cset(e^{at},\bar{X}(t))$ and let $\sigma\in\G$.
Obviously, by \eqref{eqn:solve3}-\eqref{eqn:solve6}, $\sigma(\Phi_j(t))=\Phi_j(t)$, $j\ge 3$.
Let $\Phi_2(t)=(\delta S(t),\delta E(t),\delta I(t),\delta X(t),\delta Y(t),\delta Z(t))^\T$.
Since
\[
\frac{\d}{\d t}\sigma(te^{-at})=-b\sigma(te^{-bt})+e^{-at},
\]
we have
\[
\sigma(te^{-at})=(t+c)e^{-at},
\]
where $c\in\Cset$.
From \eqref{eqn:solve2b} we have
\begin{align*}
&
\sigma(\delta I(t))=a((t+c)e^{-at})=\delta I(t)+ace^{-at},\\
&
\sigma(\delta X(t))=\frac{r}{a}((a(t+c)+1)e^{-at}-1)\bar{X}(t)=\delta X(t)+rce^{-at}\bar{X}(t)
\end{align*}
while $\sigma(\delta E(t))=\delta E(t)$ and $\sigma(\delta Y(t))=\delta Y(t)$.
Since $t\not\in\Kset$, we have $c\neq 0$ for some $\sigma\in\G$.
Thus, we see that
\[
\sigma(\Phi_2(t))=\Phi_2(t)+ac\Phi_3(t)+\frac{rc}{C_3}\Phi_4(t).
\]
We have the following for $\Phi_1(t)$ as in Lemma~\ref{lem:4b}.

\begin{lem}
\label{lem:5a}
Let $\sigma\in\G$.
We have
\[
\sigma(\Phi_1(t)) =\Phi_1(t)+\alpha_1\Phi_2(t)+\alpha_2\Phi_3(t)
 +\alpha_3\Phi_4(t)+\alpha_4\Phi_5(t)+\alpha_5\Phi_6(t)
\]
for some $\alpha_1,\alpha_2,\alpha_3,\alpha_4,\alpha_5\in\Cset$.
\end{lem}

\begin{proof}
Let $\Phi_2(t)=(\delta S(t),\delta E(t),\delta I(t),\delta X(t),\delta Y(t),\delta Z(t))^\T$.
Obviously, $\sigma(\delta S(t))=\delta S(t)$ for any $\sigma\in\G$.
From \eqref{eqn:ve} and \eqref{eqn:solve1}
 we obtain \eqref{eqn:dE} as in the proof of Lemma~\ref{lem:4b}.
Moreover, from \eqref{eqn:ve} and \eqref{eqn:dE} we have
\[
\frac{\d}{\d t}\sigma(\delta I(t))
 =a(\delta E(t)+\alpha_1e^{-at})-a\sigma(\delta I(t)),
\]
so that by \eqref{eqn:solve1}
\begin{align}
\sigma(\delta I(t))
=&ae^{-at}\int_0^t(\delta E(\tau)e^{a\tau}+\alpha_1)\d\tau+\alpha_2e^{-at}\notag\\
=&\delta I(t)+a\alpha_1 te^{-at}+\alpha_2e^{-at},
\label{eqn:dI2}
\end{align}
where $\alpha_2\in\Cset$ is a constant.
Similarly, by \eqref{eqn:ve} and \eqref{eqn:dI2}
\[
\frac{\d}{\d t}\sigma(\delta X(t))
 =-r\bar{X}(t)(\delta I(t)+a\alpha_1te^{-at}+\alpha_2e^{-at})-\bar{I}(t)\sigma(\delta X(t)),
\]
so that by \eqref{eqn:solve1}
\begin{align*}
\sigma(\delta X(t))
=&\bar{X}(t)\left(-r\int_0^t(\delta I(\tau)+\alpha_1te^{-a\tau}+\alpha_2e^{-a\tau})\d\tau
 +\tilde{\alpha}_3\right)\\
=&\delta X(t)+\left(\frac{r\alpha_1}{a}((at+1)e^{-at}-1)
 +\frac{r\alpha_2}{a}(e^{-at}-1)+\tilde{\alpha}_3\right)\bar{X}(t),
\end{align*}
where $\tilde{\alpha}_3\in\Cset$ is a constant.
In addition, by \eqref{eqn:ve}, \eqref{eqn:solve1} and \eqref{eqn:dE}
\[
\sigma(\delta Y(t))=\delta Y(t)+\frac{r\alpha_1}{a}(e^{-at}-1)+\alpha_4,\quad
\sigma(\delta Z(t))=\delta Z(t)+\alpha_5,
\]
where $\alpha_4,\alpha_5\in\Cset$ are constants.
Letting $\alpha_3=\tilde{\alpha}_3/C_3$,
 we obtain the desired result.
\end{proof}

Thus, we have
\[
\sigma(\Phi(t))=\Phi(t)
\begin{pmatrix}
1 &  0 & 0 & 0 & 0 & 0\\
\alpha_1 & 1 & 0 & 0 & 0 & 0\\
\alpha_2 & c & 1 & 0 & 0 & 0\\
\alpha_3 & \gamma_1c & 0 & 1 & 0 & 0\\
\alpha_4 & 0 & 0 & 0 & 1 & 0\\
\alpha_5 & 0 & 0 & 0 & 0 & 1
\end{pmatrix},
\]
where $ac$ is replaced by $c$ and $\gamma_1=r/aC_3$.

\begin{lem}
\label{lem:5b}
There exists $\sigma\in\G$ such that each of the following conditions holds$\,:$
\begin{enumerate}
\setlength{\leftskip}{-1.8em}
\item[\rm(i)]
$\alpha_1=0;$
\item[\rm(ii)]
$\alpha_1=0$ and $c\neq 0$.
\end{enumerate}
\end{lem}

\begin{proof}
Let $\Phi_1(t)=(\delta S(t),\delta E(t),\delta I(t),\delta X(t),\delta Y(t),\delta Z(t))^\T$,
 as in the proof of Lemma~\ref{lem:5a}.
We can prove part~(i) as in Lemma~\ref{lem:4c}.
In particular,
\[
\delta E(t)=rC_2C_3e^{-at}\int_0^t \exp\left(\frac{rC_2}{a}e^{-at}\right)\d\tau.
\]
We turn to part~(ii).
Let $\hat{\Kset}=\Kset(\delta E(t))$.
We have $\mathrm{Gal}(\Lset/\hat{\Kset})\subset\G$,
 as in the proof of Lemma~\ref{lem:4c}.
By \eqref{eqn:dE} $\alpha_1=1$
 for any $\sigma\in\mathrm{Gal}(\Lset/\hat{\Kset})\subset\G$.
On the other hand, we have $c\neq 0$ for some $\sigma\in\mathrm{Gal}(\Lset/\hat{\Kset})$
 since $t\not\in\hat{\Kset}$.
This completes the proof.
\end{proof}

Let
\[
A=
\begin{pmatrix}
1 &  0 & 0 & 0 & 0 & 0\\
\alpha_1 & 1 & 0 & 0 & 0 & 0\\
\alpha_2 & c & 1 & 0 & 0 & 0\\
\alpha_3 & \gamma_1c & 0 & 1 & 0 & 0\\
\alpha_4 & 0 & 0 & 0 & 1 & 0\\
\alpha_5 & 0 & 0 & 0 & 0 & 1
\end{pmatrix}
\]
with $c\neq 0$.
We compute
\[
A^2=
\begin{pmatrix}
1 &  0 & 0 & 0 & 0 & 0\\
2\alpha_1 & 1 & 0 & 0 & 0 & 0\\
\alpha_1c+2\alpha_2 & 2c & 1 & 0 & 0 & 0\\
\gamma_1\alpha_1+2\alpha_3 & 2\gamma_1c & 0 & 1 & 0 & 0\\
2\alpha_4 & 0 & 0 & 0 & 1 & 0\\
2\alpha_5 & 0 & 0 & 0 & 0 & 1
\end{pmatrix}
\]
and show by induction that for any $k\in\Nset$
\begin{align*}
A^k=&
\begin{pmatrix}
1 & 0 & 0 & 0 & 0 & 0\\
k\alpha_1 & 1 & 0 & 0 & 0 & 0\\
\tfrac{1}{2}k(k-1)\alpha_1c+k\alpha_2 & kc & 0 &1 & 0 & 0\\[1ex]
\tfrac{1}{2}k(k-1)\alpha_1\gamma_1c+k\alpha_2 & k\gamma_1c & 0 &1 & 0 & 0\\[1ex]
k\alpha_4 & 0 & 0 & 0 & 1 & 0\\
k\alpha_5 & 0 & 0 & 0 & 0 & 1
\end{pmatrix}.
\end{align*}
Hence, by Lemmas~\ref{lem:5a} and \ref{lem:5b}, we see that
\begin{align*}
\G^0\supset
&\left\{\left.
\begin{pmatrix}
1 & 0 & 0 & 0 & 0 & 0\\
0 & 1 & 0 & 0 & 0 & 0\\
\ast & c & 1 & 0 & 0 & 0\\
\ast & \gamma_1c & 0 & 1 & 0 & 0\\
\ast & 0 & 0 & 0 & 1 & 0\\
\ast & 0 & 0 & 0 & 0 & 1\end{pmatrix}\,\right|\,c\in\Cset
\right\}\\
&\cup\left\{\left.
\begin{pmatrix}
1 & 0 & 0 & 0 & 0 & 0\\
\alpha_1c & 1 & 0 & 0 & 0 & 0\\
\ast & c & 1 & 0 & 0 & 0\\
\ast & \gamma_1c & 0 & 1 & 0 & 0\\
\ast & 0 & 0 & 0 & 1 & 0\\
\ast & 0 & 0 & 0 & 0 & 1
\end{pmatrix}\,\right|\,c\in\Cset
\right\},
\end{align*}
where $\alpha_1\neq 0$ is a constant
 and the asterisk represents certain constants which may be zero.
Thus, $\G^0$ is not commutative.
Using Theorem~\ref{thm:MR} and Proposition~\ref{prop:3a},
 we obtain the desired result for $b=a$.
\qed

\section*{Acknowledgments}
The author thanks Juan J. Morales-Ruiz  for useful comments
 and Shoya Motonaga for helpful discussions.
This work was partially supported by the JSPS KAKENHI Grant Number JP17H02859.


\appendix

\renewcommand{\theequation}{\Alph{section}.\arabic{equation}}


\end{document}